\documentclass[a4paper,12pt]{article}
\usepackage{amsmath, amsfonts, amssymb, amsthm}
\usepackage[english]{babel}

\newcounter{num}[section]

\newenvironment{theorem}
{\refstepcounter{num}%
\bigskip\noindent\nopagebreak[4]{\bf Theorem~\arabic{section}.\arabic{num}. }\it}

\newenvironment{proposition}
{\refstepcounter{num}%
\bigskip\noindent\nopagebreak[4]{\bf Proposition~\arabic{section}.\arabic{num}. }\it}

\newenvironment{corollary}
{\refstepcounter{num}%
\bigskip\noindent\nopagebreak[4]{\bf Corollary~\arabic{section}.\arabic{num}. }\it}

\newenvironment{lemma}
{\refstepcounter{num}%
\bigskip\noindent\nopagebreak[4]{\bf Lemma~\arabic{section}.\arabic{num}. }\it}

\newenvironment{example}
{\refstepcounter{num}%
\bigskip\noindent\nopagebreak[4]{\bf Example~\arabic{section}.\arabic{num}. }}

\newcommand{\LL}{{\mathcal{L}}}

\newcommand{\Ss}{{\mathbf{S}}}
\newcommand{\V}{{\mathrm{V}}}

\newcommand{\pr}{{\prime}}

\newcommand{\M}{{\mathcal{M}}}
\newcommand{\T}{{\mathcal{T}}}
\renewcommand{\P}{{\mathbf{P}}}
\renewcommand{\c}{{\mathbf{c}}}

\newcommand{\one}{{\mathbf{1}}}

\begin{document}

\author{Artem N. Shevlyakov}
\title{On disjunctions of equations over finite simple semigroups}

\maketitle

\abstract{A semigroup $S$ is called an equational domain if any finite union of algebraic sets over $S$ is algebraic. For a finite simple semigroup we find necessary and sufficient conditions to be an equational domain. Moreover, we study semigroups with nontrivial center and prove that any such semigroup is not an equational domain.}

\section{Introduction}

It is a well-known fact of commutative algebra that the union of two algebraic sets over a field $k$ (i.e. sets defined by system of polynomial equations) is algebraic. 

However, the definition of an equation in variables $X=\{x_1,x_2,\ldots,x_n\}$ can be defined over an arbitrary algebraic structure, not merely over a field. For example, an equation over a group $G$ is an expression $w(X)=1$, where $w(X)$ is an element of the free product $G\ast F(X)$ (\cite{AG_over_groupsI,AG_over_groupsII}), i.e. $w(X)$ is a product of integer degrees of variables and elements of the group $G$. An algebraic set over a group $G$ is defined as a solution set of a system of equations.

There exist groups, where any finite union of algebraic sets is algebraic. Following~\cite{uniTh_IV}, groups with such property are called {\it equational domains}, and in~\cite{uniTh_IV} these groups were completely described. 

It follows from~\cite{uniTh_IV} that the next groups are equational domains:
\begin{enumerate}
\item free non-abelian groups (this case was earlier proved by G.~Gurevich, see the proof in~\cite{makanin});
\item simple non-abelian groups (this case also follows from~\cite{rhodes}). 
\end{enumerate}

After the complete description of equational domains among groups it is natural to pose the similar problem for semigroups. Before to formulate this problem let us give some definitions of algebraic geometry over semigroups. 

All definitions below follow from the papers~\cite{uniTh_I,uniTh_II}, where  such notions were given for an arbitrary algebraic structure in a language with no predicates.

An equation over a semigroup $S$ is an equality $t(X)=s(X)$, where each part is a product consisting variables from the set $X$ and elements of the semigroup $S$. Using the definition of equation, one can naturally give the definitions of algebraic set and equational domain for a semigroup $S$.

\medskip 

{\bf Problem.} {\it Is there a nontrivial semigroup $S$ such that
\begin{enumerate}
\item $S$ is an equational domain;
\item $S$ is not a group.
\end{enumerate}}

Notice that the second condition of the problem above is essential, since one can find a group which is an equational domain as a semigroup. Indeed, take a finite group $G$ which is an equational domain (for example, $G=A_5$). By the choice of the group $G$, any finite union of algebraic sets $Y_1\cup Y_2\cup\ldots Y_n$ is a solution set of a system of equations $\Ss=\{w_i(X)=1|1\leq i\leq m\}$, where the group words $w_i(X)\in F(X)\ast G$ may contain variables in negative degrees (a constant in negative degree can be calculated and replaced to the another element of $G$). As the group $G$ is finite, one can replace all negative degrees to positive ones by the law: 
\[x^{-1}=x^{|G|-1}.\]  
Thus, $\Ss$ becomes a system of semigroup equations, hence the group $G$ is an equational domain as a semigroup.  

\bigskip

Let us explain the results of our paper. We study the class of finite simple groups, i.e. semigroups with no proper two-sided ideals. In Theorem~\ref{th:main} we give the necessary and sufficient conditions for a finite simple semigroup $S$ to be an equational domain. According Theorem~\ref{th:main}, it is easily defined a semigroup which solves the problem above positively (Example~\ref{ex:domain_240}). 

In Theorem~\ref{th:about_nonsimple_domains} we prove that the adjunction of the identity element to a finite simple semigroup preserves the property ``to be an equational domain''. Thus, Theorem~\ref{th:about_nonsimple_domains} solves the problem above in the class of non-simple semigroups.

Following~\cite{uniTh_IV}, any group with nontrivial center is not an equational domain. In Section~\ref{sec:center_semigroups} we prove the similar result for semigroups (Theorem~\ref{th:center_semigroups}).

\section{Basics: semigroups}

Let us give the general definitions of semigroup theory. For more details see~\cite{lyapin,skornyakov}. 

{\it A semigroup} is a nonempty set with associative binary operation $\cdot$ which is called a multiplication. A semigroup with a single element is called \textit{trivial}.

Elements $a,b\in S$ \textit{commute} if it holds $ab=ba$. A semigroup is commutative, if any pair of its elements commute. \textit{A center} of a semigroup $S$ consists of all elements $z\in S$ which commute with any element of $S$.    

An element $e$ ($0$) of a semigroup $S$ is an \textit{identity element} (\textit{zero}) if for any $s\in S$ we have $se=es=s$ ($s0=0s=0$). Clearly, the identity element and zero (if they exist) belong to the center of a semigroup.    

A subset $I\subseteq S$ is called a \textit{left (right) ideal} if for any $s\in S$, $a\in I$ it holds $sa\in I$ ($as\in I$). An ideal which is right and left simultaneously is said to be {\it two-sided} (or {\it an ideal} for shortness). For example, the set $\{0\}$ in a semigroup with zero is always ideal.  

A semigroup $S$ with a unique ideal $I=S$ is called \textit{simple}. Let us give the next theorem in the form proven in~\cite{skornyakov}.  

\begin{theorem}
\label{th:sushkevic_rees}
For any finite simple semigroup $S$ there exists a finite group $G$ and finite sets $I,\Lambda$ such that  $S$ is isomorphic to the set of triples $(\lambda,g,i)$, $g\in G$, $\lambda\in\Lambda$, $i\in I$. The multiplication over the triples $(\lambda,g,i)$ is defined by
\[
(\lambda,g,i)(\mu,h,j)=(\lambda,gp_{i\mu}h,j),
\]
where $p_{i\mu}\in G$ is an element of a matrix $\P$ such that
\begin{enumerate}
\item $\P$ consists of  $|I|$ rows and $|\Lambda|$ columns;
\item the elements of the first row and the first column equal  $1\in G$ (i.e. $\P$ is {\it normalized}).
\end{enumerate}
\end{theorem}

Following Theorem~\ref{th:sushkevic_rees}, we denote any finite simple semigroup $S$ by $S=(G,\P,\Lambda,I)$. Notice that the cardinalities of the sets $\Lambda$, $I$ are equal respectively to the numbers of minimal right and left ideals of a semigroup $S$.

\begin{corollary}
\label{cor:when_is_group}
A finite simple semigroup $S=(G,\P,\Lambda,I)$ is a group iff $|\Lambda|=|I|=1$.
\end{corollary}

\medskip

The numbers $\lambda\in\Lambda$, $i\in I$ of an element $(\lambda,g,i)$ in a finite simple semigroup $S=(G,\P,\Lambda,I)$ is said to be the {\it first} and the {\it second} index respectively.

The minimal ideal (in finite semigroups it always exists) of a semigroup $S$ is called a \textit{kernel} and denoted by $Ker(S)$. Obviously, if $S=Ker(S)$ the semigroup is simple. If $Ker(S)$ is a group then $S$ is said to be a \textit{homogroup}. The next theorem contains the necessary information about homogroups.

\begin{theorem}
\label{th:homogroups_properties}
In a homogroup $S$ the identity element $e$ of the kernel $Ker(S)$ is idempotent (i.e. $e^2=e$) and belongs to the center of $S$.
\end{theorem}

\section{Basics: algebraic geometry}

All definitions below are derived from the general notions of~\cite{uniTh_I,uniTh_II}, where the definitions of algebraic geometry were formulated for an arbitrary algebraic structure in the language with no predicates.

Semigroups as algebraic structures are often considered in the language $\LL_0=\{\cdot\}$. However, for a given semigroup $S$ one can add to the language $\LL_0$ the set of constants $\{s|s\in S\}$ which corresponds to all elements of the semigroup $S$. We denote the extended language by $\LL_S$, and further we consider all semigroups in such language.

Let $X$ be a finite set of variables  $x_1,x_2,\ldots,x_n$. \textit{An $\LL_S$-term} in the variables  $X$ is a finite product of variables and constants $s\in S$. For example, the following expressions $xsy^2x$, $xs_1ys_2x^2$, $x^2yxz$ are $\LL_S$-terms.

{\it An equation} over $\LL_S$ is an equality of two $\LL_S$-terms $t(X)=s(X)$. {\it A system of equations} over $\LL_S$ ({\it a system} for shortness) is an arbitrary set of equations over $\LL_S$.
 
A point $P=(p_1,p_2,\ldots,p_n)\in S^n$ is a \textit{solution} of a system $\Ss$ in variables $x_1,x_2,\ldots,x_n$, if the substitution $x_i=p_i$ reduces any equation of $\Ss$ to a true equality in the semigroup $S$. The set of all solutions of a system $\Ss$ in the semigroup $S$ is denoted by $\V_S(\Ss)$. A set $Y\subseteq S^n$ is called  {\it algebraic} over the language $\LL_S$ if there exists a system over $\LL_S$ in variables $x_1,x_2,\ldots,x_n$ with the solution set $Y$. 

Following~\cite{uniTh_IV}, let us give the main definition of our paper.

A semigroup $S$ is an {\it equational domain} ({\it e.d.} for shortness) in the language $\LL_S$ if for any finite set of algebraic sets $Y_1,Y_2,\ldots,Y_n$ over $\LL_S$ the union $Y=Y_1\cup Y_2\cup\ldots\cup Y_n$ is algebraic. 

Any single point $P=\{(p_1,p_2,\ldots,p_n)\}\subseteq S^n$ is algebraic set over any semigroup $S$ in the language $\LL_S$, since it equals to the solution of the system $\Ss_P(x_1,x_2,\ldots,x_n)=\{x_i=p_i|1\leq i\leq n\}$. As in a finite semigroup $S$ any set $M\subseteq S^n$ is a finite union of points, we obtain the next simple proposition.

\begin{proposition}
A finite semigroup $S$ is an e.d. in the language $\LL_S$ iff for any natural number $n$ every set $M\subseteq S^n$ is algebraic over $\LL_S$.
\end{proposition}

The next theorem contains the necessary and sufficient conditions for a semigroup to be an e.d.  

\begin{theorem}\textup{\cite{uniTh_IV}}
\label{th:about_M}
A semigroup $S$ in the language $\LL_S$ is an e.d. iff the set 
\[
\M_{sem}=\{(x_1,x_2,x_3,x_4)|x_1=x_2\mbox{ or }x_3=x_4\}\subseteq S^4
\]
is algebraic, i.e. there exists a system $\Ss$ in the variables $x_1,x_2,x_3,x_4$ with the solution set $\M_{sem}$.
\end{theorem}

Below we will study equations over groups, therefore we give some definitions of algebraic geometry over groups. Any group $G$ below will be considered in the language $\LL_G=\{\cdot,^{-1},1\}\cup\{g|g\in G\}$ extended by the constants $\{g|g\in G\}$. \textit{An $\LL_G$-term}  in the variables  $X=\{x_1,x_2,\ldots,x_n\}$ is a finite product which consists of the variables in integer degrees and constants $g\in G$. In other words, an $\LL_G$-term is an element of the free product $F(X)\ast G$, where $F(X)$ is a free group generated by the set $X$. 

The definitions of equations, algebraic sets and equational domains over groups are similar to the corresponding definitions in the semigroup case. 

For the groups of the language $\LL_G$ we have the following result.

\begin{theorem}\textup{\cite{uniTh_IV}}
A group $G$ in the language $\LL_G$ is an e.d. iff the set 
\label{th:criterion_for_groups}
\begin{equation*}
\M_{gr}=\{(x_1,x_2)|x_1=1\mbox{ or }x_2=1\}\subseteq G^2
\end{equation*}
is algebraic, i.e. there exists a system $\Ss$ in variables $x_1,x_2$ with the solution set $\M_{gr}$.
\end{theorem}

One can reformulate Criterion~\ref{th:criterion_for_groups} in more simple form using the next definition. An element $x\neq 1$ of a group $G$ is a \textit{zero-divisor} if there exists $1\neq y\in G$ such that for any $g\in G$ it holds $[x,y^g]=1$ (here $y^g=gyg^{-1}$, $[a,b]=a^{-1}b^{-1}ab$).

\begin{theorem}\textup{\cite{uniTh_IV}}
\label{th:zero_divisors}
A group $G$ in the language $\LL_G$ is an e.d. iff it does not contain zero-divisors.
\end{theorem}

\section{Finite simple semigroups}

Let $S=(G,\P,\Lambda,I)$ be a finite simple semigroup. Define a set
\begin{equation*}
\Gamma=\{(1,g,1)|g\in G\}\subseteq S.
\end{equation*}

It is easy to prove that $\Gamma$ is isomorphic to the group $G$ with the identity element $(1,1,1)$. 

We say that the matrix $\P$ is {\it nonsingular} if it does not contain two equal rows or columns. 

Let $M$ be a subset of $S^n$. By $\T(M,\Gamma)$ denote the set of all $\LL_S$-terms in variables $x_1,x_2,\ldots,x_n$ whose values belong to the subgroup $\Gamma$ for all points $P\in M$. For example, $t(x)=(1,g,2)x(3,h,1)\in\T(S,\Gamma)$, $s(x,y)=(1,g,1)x^2(3,h,4)y(2,f,1)\in\T(S^2,\Gamma)$. The next statement holds. 

\begin{lemma}
Any term of the set $\T(S^n,\Gamma)$ should begin with a constant of the form $(1,g,k)$ and end by  $(\lambda,h,1)$ for some $\lambda,h$.
\label{l:term_from_T_must_begin}
\end{lemma}

\medskip

Let $s_1,s_2$ be two distinct elements of a semigroup $S$. We say that an $\LL_S$-term {\it separates} the elements $s_1,s_2$ if $t(s_1)\neq t(s_2)$.

\begin{lemma}
\label{l:exists_dist_term}
Suppose for a finite simple semigroup $S=(G,\P,\Lambda,I)$ the matrix $\P$ is nonsingular. Then for any pair of distinct elements $s_1,s_2\in S$ there exists a term  $t(x)\in \T(S,\Gamma)$ separating $s_1,s_2$.
\end{lemma}
\begin{proof}
Consider two cases.
\begin{enumerate}
\item Let $s_1=(\lambda,g,i)$, $s_2=(\mu,h,j)$, $h\neq g$. Put $t(x)=(1,1,1)x(1,1,1)$. We have 
\[
t(s_1)=(1,1,1)(\lambda,g,i)(1,1,1)=(1,p_{1\lambda}gp_{i1},1)=(1,g,1),
\]
\[
t(s_2)=(1,1,1)(\mu,h,j)(1,1,1)=(1,p_{1\mu}hp_{j1},1)=(1,h,1),
\]
and hence $t(s_1)\neq t(s_2)$.
\item Let $s_1=(\lambda,g,i)$, $s_2=(\mu,g,j)$ and $\lambda\neq\mu$ (the proof of the case $i\neq j$ is similar). Assume that all terms of the form $t(x)=(1,1,k)x(1,1,1)\in\T(S,\Gamma)$ do not separate   $s_1,s_2$. In other words, the elements
\[
(1,1,k)(\lambda,g,i)(1,1,1)=(1,p_{k\lambda}g,1)(1,1,1)=(1,p_{k\lambda}gp_{11},1)=(1,p_{k\lambda}g,1)
\]
and 
\[
(1,1,k)(\mu,g,j)(1,1,1)=(1,p_{k\mu}g,1)(1,1,1)=(1,p_{k\mu}gp_{11},1)=(1,p_{k\mu}g,1)
\]
equal to each other for any $k\in I$. Thus, it holds $p_{k\lambda}=p_{k\mu}$ for all $k\in I$. It means that the columns with indexes $\lambda,\mu$ are the same, hence the matrix $\P$ is singular. We came to the contradiction.
\end{enumerate}
\end{proof}

\begin{lemma}
\label{l:exists_2_non_dist_elems}
Suppose the matrix $\P$ of a finite simple semigroup $S=(G,\P,\Lambda,I)$ has equal rows (columns) with indexes $i,j$ ($\lambda,\mu$). Then for the elements $s_1=(1,1,i)$, $s_2=(1,1,j)$ ($s_1=(\lambda,1,1)$, $s_2=(\mu,1,1)$) and for an arbitrary $\LL_S$-term $t(x)$ one of the following conditions holds:
\begin{enumerate}
\item $t(s_1)=t(s_2)$;
\item $t(s_1)=(\nu,g,i)$, $t(s_2)=(\nu,g,j)$ for some $g\in G$, $\nu\in\Lambda$ if $t(x)$ ends on the variable $x$ ($t(s_1)=(\lambda,g,k)$, $t(s_2)=(\mu,g,k)$ for some $g\in G$, $k\in I$ if $t(x)$ begins with $x$). 
\end{enumerate}
\end{lemma}
\begin{proof}
Let $\P$ has equal columns with indexes $\lambda,\mu$ (similarly, one can consider $\P$ with equal rows).

Let us consider an $\LL_S$-term $t(x)$ which does not begin with the variable $x$. 
\begin{equation*}
t(x)=\c_1x^{n_1}\c_2x^{n_2}\c_3\ldots \c_mx^{n_m}\c_{m+1},
\end{equation*}
where $\c_k=(\lambda_k,c_k,i_k)$.

Compute    
\begin{multline*}
t(s_1)=(\lambda_1,c_1,i_1)(\lambda,1,1)^{n_1}(\lambda_2,c_2,i_2)(\lambda,1,1)^{n_2}(\lambda_3,c_3,i_3)\ldots (\lambda_m,c_m,i_m)(\lambda,1,1)^{n_m}\\
(\lambda_{m+1},c_{m+1},i_{m+1})=(\lambda_1,c_1,i_1)((\lambda,1,1)(\lambda_2,c_2,i_2))((\lambda,1,1)(\lambda_3,c_3,i_3))\ldots (\lambda_m,c_m,i_m))\\
((\lambda,1,1)(\lambda_{m+1},c_{m+1},i_{m+1}))=
(\lambda_1,c_1,i_1)(\lambda,c_2,i_2)(\lambda,c_3,i_3)\ldots (\lambda,c_m,i_m)
(\lambda,c_{m+1},i_{m+1})=\\
(\lambda_1,c_1p_{i_1\lambda}c_2p_{i_2\lambda}c_3p_{i_3\lambda}\ldots c_mp_{i_m\lambda}c_{m+1},i_{m+1}).
\end{multline*}
Similarly,
\[
t(s_2)=(\lambda_1,c_1p_{i_1\mu}c_2p_{i_2\mu}c_3p_{i_3\mu}\ldots c_mp_{i_m\mu}c_{m+1},i_{m+1}).
\]
As the columns in $\P$ with indexes $\lambda,\mu$ are equal to each other, for any $k$ we have $p_{i_k\lambda}=p_{i_k\mu}$, hence $t(s_1)=t(s_2)$.

Consider now a term $t(x)=x^nt^\pr(x)$, where $t^\pr(x)$ does not begin with constant. Above we proved  $t^\pr(s_1)=t^\pr(s_2)=(\nu,g,k)$. Thus,
\[
t(s_1)=(\lambda,1,1)^n(\nu,g,k)=(\lambda,1,1)(\nu,g,k)=(\lambda,g,k),
\] 
\[
t(s_2)=(\mu,1,1)^n(\nu,g,k)=(\mu,1,1)(\nu,g,k)=(\mu,g,k).
\]
\end{proof}

\begin{lemma}
\label{l:singular->not_ED}
If the matrix $\P$ is singular, a finite simple semigroup $S=(G,\P,\Lambda,I)$ is not an e.d. in the language $\LL_S$.
\end{lemma}
\begin{proof}
Assume that $S$ is an e.d. with a singular matrix $\P$ which has equal columns with numbers $\lambda,\mu$ (similarly , one can consider a matrix $\P$ with equal rows). 

Let $\Ss(x,y)$ be a system with the solution set $\M=\{(x,y)|x=(\lambda,1,1)\mbox{ or }y=(\lambda,1,1)\}$, and $t(x,y)=s(x,y)$ is an equation of $\Ss$ such that $t((\mu,1,1),(\mu,1,1))\neq s((\mu,1,1),(\mu,1,1))$. 

Assume that the terms $t(x,(\mu,1,1))$, $s(x,(\mu,1,1))$ do not begin with the variable $x$. Hence, Lemma~\ref{l:exists_2_non_dist_elems} gives the equalities
\[
t((\lambda,1,1),(\mu,1,1))=t((\mu,1,1),(\mu,1,1)),\; s((\lambda,1,1),(\mu,1,1))=s((\mu,1,1),(\mu,1,1)).
\]
As 
\[
t((1,\lambda,1),(\mu,1,1))=s((\lambda,1,1),(\mu,1,1)),
\]
we have 
\[
t((1,\mu,1),(\mu,1,1))=s((\mu,1,1),(\mu,1,1)),
\]
that contradicts with the choice of the equation $t(x,y)=s(x,y)$.

Assume now that the both terms $t(x,(\mu,1,1))$, $s(x,(\mu,1,1))$ begin with $x$. Using Lemma~\ref{l:exists_2_non_dist_elems}, we obtain
\[
t((\lambda,1,1),(\mu,1,1))=(\lambda,g,k),\;t((\mu,1,1),(\mu,1,1))=(\mu,g,k).
\]
Since $((\lambda,1,1),(\mu,1,1))\in \M$, then $s((\lambda,1,1),(\mu,1,1)=(\lambda,g,k)$. According Lemma~\ref{l:exists_2_non_dist_elems}, the values of the term $s(x,(\mu,1,1))$ are
\[
s((\lambda,1,1),(\mu,1,1))=(\lambda,g,k),\;s((\mu,1,1),(\mu,1,1))=(\mu,g,k).
\]
Therefore,
\[
t((1,\mu,1),(\mu,1,1))=s((\mu,1,1),(\mu,1,1)),
\]
that contradicts with the choice of the equation $t(x,y)=s(x,y)$.

Thus, the last case is: the term $t(x,(\mu,1,1))$ is $t(x,(\mu,1,1))=xt^\pr(x,(\mu,1,1))$, and $s(x,(\mu,1,1))$ begins with a constant $\c$. We have exactly two possibilities:
\begin{enumerate}
\item the constant $\c$ was obtained by by the substitution of the element $(\mu,1,1)$ instead of the variable $y$; in other words, the term $s(x,y)$ is the expression $ys^\pr(x,y)$;
\item the constant $\c$ occurs in $s(x,y)$, i.e. $s(x,y)=\c s^\pr(x,y)$.
\end{enumerate}  
Let us show that the both cases above are impossible.
\begin{enumerate}
\item The equation $xt^\pr(x,y)=ys^\pr(x,y)$ does not satisfy the point $((\lambda,1,1),(\mu,1,1))\in\M$, as the element $t((\lambda,1,1),(\mu,1,1))$ has the first index $\lambda$, but the first index of the element $s((\lambda,1,1),(\mu,1,1))$ is $\mu$.
\item Let $\c=(\nu,g,k)$, hence the equation $xt^\pr(x,y)=\c s^\pr(x,y)$ does not satisfy either $((\lambda,1,1),(\lambda,1,1))\in\M$ (if $\nu\neq\lambda$) or $((\mu,1,1),(\lambda,1,1))\in\M$ (if $\nu\neq\mu$).
\end{enumerate}
\end{proof}

\begin{lemma}
\label{l:about_equiv_over_Gamma}
Let $S=(G,\P.\Lambda,I)$ be a finite simple semigroup, and $x,y\in\Gamma$. Hence
\begin{enumerate}
\item $x(\lambda,c,i)y=x(1,c,1)y$;
\item if an equation 
\begin{equation}
(\lambda,c,i)t(x,y)=(\lambda^\pr,c^\pr,i^\pr)t^\pr(x,y)
\label{eq:rrrrr}
\end{equation}
is consistent over $S$, then it is equivalent to   
\[
(1,c,1)t(x,y)=(1,c^\pr,1)t^\pr(x,y)
\]
over the group $\Gamma$;
\item is an equation 
\begin{equation*}
t(x,y)(\lambda,c,i)=t^\pr(x,y)(\lambda^\pr,c^\pr,i^\pr)
\end{equation*}
is consistent over $S$, then it is equivalent to    
\[
t(x,y)(1,c,1)=t^\pr(x,y)(1,c^\pr,1)
\]
over the group $\Gamma$.
\end{enumerate}
\end{lemma}
\begin{proof}
The proof of the first statement is straightforward using the equalities  $p_{1\lambda}=p_{i1}=1$ for all $\lambda\in\Lambda$, $i\in I$.

Let us prove the second statement. By the consistency of the equation~(\ref{eq:rrrrr}) we have $\lambda=\lambda^\pr$. Without loss of generality one can state that the terms $t(x,y), t^\pr(x,y)$ begin with a variable. Hence, the values of the terms $t(x,y), t^\pr(x,y)$ at a point 
\[
(x,y)\in\Gamma^2\cap\V_S((\lambda,c,i)t(x,y)=(\lambda^\pr,c^\pr,i^\pr)t^\pr(x,y))
,
\]
equal $(1,g,j)$, $(1,g^\pr,j)$ respectively. Thus, we have  $(\lambda,c,i)(1,g,j)=(\lambda,c^\pr,i)(1,g^\pr,j)$, that equivalent to $(\lambda,cg,j)=(\lambda,c^\pr g^\pr,j)$, and therefore $cg=c^\pr g^\pr$.

If we put  $\lambda=1$, $i=i^\pr=1$, it is easy to see that the point $(x,y)$ remains the solution of the equation~(\ref{eq:rrrrr}).

The third statement of the lemma is similar to the second one.
  
\end{proof}

\begin{lemma}
\label{l:S-eq_dom->G-eq_dom}
If a finite simple semigroup $S=(G,\P,\Lambda,I)$ is an e.d. in the language $\LL_S$, then the group $G$ is an e.d. in the group language $\LL_G$. 
\end{lemma}
\begin{proof}
As $S$ is an e.d., the set $\M=\{(x,y)|x=(1,1,1)\mbox{ or }y=(1,1,1)\}\subseteq S^2$ is algebraic over $S$. In other words, there exists a system $\Ss$ over $\LL_S$ with $\V_S(\Ss)=\M$. 

If we assume that there is $(\lambda,g,i)t^\pr(x,y)=xs^\pr(x,y)\in\Ss$, then this equation does not satisfy the point $((\mu,h,j),(1,1,1))\in\M$, where $\mu\neq\lambda$.

Thus, for any equation from $\Ss$  one can apply the formulas from Lemma~\ref{l:about_equiv_over_Gamma}, and obtain a system $\Ss^\pr$ whose constants belong to the group $\Gamma$. Moreover, the system $\Ss^\pr$ is equivalent to $\Ss$ over the group $\Gamma$.

Finally, we have $\V_\Gamma(\Ss^\pr)=\{(x,y)|x=(1,1,1)\mbox{ or }y=(1,1,1)\}\subseteq\Gamma^2$, and, by Theorem~\ref{th:criterion_for_groups}, the group $\Gamma$ is an equational domain in the language  $\LL_\Gamma$. The isomorphism between the groups $\Gamma,G$ proves the lemma.
\end{proof}

Let $P=(p_1,p_2,\ldots,p_n)\in S^n$. By $\T_P(M,\Gamma)$ (where $P\in M\subseteq S^n$) denote the set of all terms $t(X)\in\T(S^n,\Gamma)$ such that $t(P)\neq(1,1,1)$, $t(Q)=(1,1,1)$ for all $Q\in M\setminus\{P\}$.

\begin{lemma}
\label{l:sufficient_conditions}
Let $S=(G,\P,\Lambda,I)$ be a finite simple semigroup, $\P$ is nonsingular, and $G$ is an e.d. in the language $\LL_G$. Then for any natural $n$ and any point $P=(p_1,p_2,\ldots,p_n)\in S^n$ the set $\T_P(S^n,\Gamma)$ 
\begin{enumerate}
\item is nonempty;
\item contains all terms of the form $(1,g,1)t(X)(1,g^{-1},1)$, $g\in G$ if $t(X)\in \T_P(S^n,\Gamma)$.
\end{enumerate}
\end{lemma}
\begin{proof}
The second property follows from the first one: let $t(X)\in\T_P(S^n,\Gamma)$, hence
\[
(1,g,1)t(Q)(1,g^{-1},1)=(1,g,1)(1,1,1)(1,g^{-1},1)=(1,g1g^{-1},1)=(1,1,1),
\]
\[
(1,g,1)t(P)(1,g^{-1},1)=(1,g,1)(1,h,1)(1,g^{-1},1)=(1,ghg^{-1},1)\neq(1,1,1),\mbox{ since }h\neq 1.
\]

Let us prove $\T_P(S^n,\Gamma)\neq\emptyset$.

Further we shall use the denotation:
\[t^{-1}(X)=t^{|G|-1}(X).\]
Obviously, that for any term $t(X)\in\T(S^n,\Gamma)$ it holds $t^{-1}(X)\in\T(S^n,\Gamma)$, and
\[
t(X)t^{-1}(X)=t^{-1}(X)t(X)=t^{|G|}(X)=(1,1,1)\mbox{ for all }X\in S^n.
\]

We prove $\T_P(M,\Gamma)\neq\emptyset$ by the induction on the cardinality of the set $M\subseteq S^n$. Let $|M|=2$, and $P,Q\subseteq S^n$ two distinct points of the set $M$.

Without loss of generality one can assume that the points $P,Q$ have the distinct first coordinates $p_1\neq q_1$. By Lemma~\ref{l:exists_dist_term}, there exists a term $t(x)\in\T(S,\Gamma)$ with $t(p_1)\neq t(q_1)$. Let $s(X)=t(x_1)t^{-1}(q_1)\in\T(S,\Gamma)$, and we have $s(P)=t(p_1)t^{-1}(q_1)\neq (1,1,1)$, $s(Q)=t(q_1)t^{-1}(q_1)=(1,1,1)$. Thus, $s(X)\in \T_P(M,\Gamma)$.  

Suppose that for any set $M$ with $|M|\leq m$ the statement of the lemma is proved. Let us prove the lemma for a set $M$ with $m+1$ elements.

Let $M=\{P,Q_1,Q_2,\ldots,Q_m\}$. By the assumption of the induction, there exist terms
\[t(X)\in\T_P(\{P,Q_2,Q_3,\ldots,Q_m\},\Gamma),s(X)\in\T_P(\{P,Q_1,Q_3,\ldots,Q_m\},\Gamma),\]
with values

\begin{tabular}{ccccccc}
&$P$&$Q_1$&$Q_2$&$Q_3$&$\ldots$&$Q_m$\\
$t(X)$&$(1,g_1,1)$&$(1,h_1,1)$&$(1,1,1)$&$(1,1,1)$&$\ldots$&$(1,1,1)$\\
$s(X)$&$(1,g_2,1)$&$(1,1,1)$&$(1,h_2,1)$&$(1,1,1)$&$\ldots$&$(1,1,1)$
\end{tabular} 

One can choose the elements $g_1,g_2\in G$ which do not commute. Indeed, the second property of the set $\T_P(M,\Gamma)$ allows us to take $g_2$ from the conjugacy class $C=\{gg_2g^{-1}|g\in G\}$. If $g_1$ commutes with all elements of $C$ then $g_1$ is a zero-divisor in the group $G$, and, by Theorem~\ref{th:zero_divisors}, $G$ is not an e.d. that contradicts with the conditions of the lemma.   

The values of the term
\[p(X)=[t(X),s(X)]=t^{-1}(X)s^{-1}(X)t(X)s(X)\in\T(S,\Gamma),
\]
are

\begin{tabular}{ccccccc}
&$P$&$Q_1$&$Q_2$&$Q_3$&$\ldots$&$Q_m$\\
$p(X)$&$(1,[g_1,g_2],1)$&$(1,1,1)$&$(1,1,1)$&$(1,1,1)$&$\ldots$&$(1,1,1)$
\end{tabular} 

where $[g_1,g_2]=g_1^{-1}g_2^{-1}g_1g_2\neq 1$ is the commutator of the elements $g_1,g_2$ in the group $G$.
 
Thus, $p(X)\in\T_P(M,\Gamma)$, and we have proved the lemma.
\end{proof}

\begin{theorem}
\label{th:main}
A finite simple semigroup $S=(G,\P,\Lambda,I)$ is an e.d. in the language $\LL_S$ iff the next two conditions hold:
\begin{enumerate}
\item $\P$ is nonsingular; 
\item $G$ is an e.d. in the group language $\LL_G$. 
\end{enumerate}
\end{theorem}
\begin{proof}
The non-singularity of the matrix $\P$ follows from Lemma~\ref{l:singular->not_ED}, and Lemma~\ref{l:S-eq_dom->G-eq_dom} states that $G$ is an e.d. in the language $\LL_G$.

Prove the converse. Consider the set $\M_{sem}=\{(x_1,x_2,x_3,x_4)|x_1=x_2\mbox{ or }x_3=x_4\}\subseteq S^4$. By Lemma~\ref{l:sufficient_conditions} for any point $P\notin \M_{sem}$ there exists a term $t_P(x_1,x_2,x_3,x_4)\in\T_P(S^4,\Gamma)$. It is clear that the solution set of the system $\Ss=\{t_P(x_1,x_2,x_3,x_4)=(1,1,1)|P\notin \M_{sem}\}$ equals $\M_{sem}$, hence $\M_{sem}$ is algebraic. By Theorem~\ref{th:about_M}, the semigroup $S$ is an e.d. in the language $\LL_S$.  
\end{proof}

\begin{corollary}
Suppose a finite simple semigroup $S$ is an e.d. in the language $\LL_S$. Then any nonempty set $M\subseteq S^n$ equals to the solution set of a system $\Ss=\{t_i(X)=(1,1,1)|1\leq i\leq m\}$, where $t_i(X)\in\T(S^n,\Gamma)$, $m=|S|^n-|M|$.  
\end{corollary}
\begin{proof}
Suppose the set $S^n\setminus M$ consists of the points $P_1,P_2,\ldots,P_m$, where $m=|S|^n-|M|$. Following Lemma~\ref{l:sufficient_conditions}, there exist terms $t_i(X)\in\T_{P_i}(S^n,\Gamma)$ such that the solution set of the equation $t_i(X)=(1,1,1)$ is $S^n\setminus\{P_i\}$. Thus, the solution set of the system $\Ss=\{t_i(X)=(1,1,1)|1\leq i\leq m\}$ coincides with $M$.
\end{proof}

\begin{corollary}
A finite simple semigroup $S=(G,\P,\Lambda,I)$ is not an e.d. in the language $\LL_S$ if at least one of the following holds
\begin{enumerate}
\item $|G|^{|\Lambda|-1}<|I|$;
\item $|G|^{|I|-1}<|\Lambda|$;
\item $|\Lambda|=1$, $|I|>1$;
\item $|I|=1$, $|\Lambda|>1$;
\item $G=\{1\}$ and at least one of the numbers $|I|,|\Lambda|$ is more than $1$.
\end{enumerate}
\end{corollary}
\begin{proof}
The last three conditions follows from the first and second statements.

Let us prove the first statement (the proof of the second one is similar).

We find the relations between the numbers $|\Lambda|,|I|$ which guarantee two equal rows in the matrix $\P$. Remind that it implies the singularity of $\P$, and, by Theorem~\ref{th:main}, $S$ is not an e.d.

The number of different rows $(1,g_2,g_3,\ldots,g_{|\Lambda|})$ equals $|G|^{|\Lambda|-1}$ (we put $g_1=1$, since the first element of any row in $\P$ is $1$). Hence, in every matrix $\P$ with at least $|G|^{|\Lambda|-1}$ rows (i.e. $|G|^{|\Lambda|-1}<|I|$) there always exist two equal rows, and $\P$ becomes singular.
\end{proof}

\begin{example}
\label{ex:domain_240}
Consider a finite simple semigroup $S_{240}$ defined by the next parameters $(G,\P,\Lambda,I)$: $G=A_5$ (the alternating group of degree $5$), $\Lambda=I=\{1,2\}$,
\[\P=\begin{pmatrix}1&1\\1&g\end{pmatrix},\]
where $1\neq g\in A_5$. The order of  $S_{240}$ is $|S_{240}|=|A_5|\cdot 2\cdot 2=240$. As $A_5$ is simple non-abelian, by~\cite{uniTh_IV} it is an e.d.. The matrix $\P$ is non-singular, hence by Theorem~\ref{th:main}, the semigroup $S_{240}$ is an e.d. in the language $\LL_S$. According Corollary~\ref{cor:when_is_group}, $S_{240}$ is not a group, thus the semigroup $S_{240}$ solves the problem posed in the introduction. 
\end{example}

\section{Equational domains among non-simple semigroups}

Let $S=(G,\P,\Lambda,I)$ be a finite simple semigroup which is not a group (i.e. $|\Lambda|>1$ or $|I|>1$). It is easy to check that $S$ does not contain the identity element. By $S^\ast=S\cup\{\one\}$ we denote the semigroup (monoid) which is obtained from $S$ by the adjunction of the identity element $\one$: $\one s=s\one=s$ for all $s\in S$. Obviously, the set $S$ is an ideal in $S^\ast$, hence $S^\ast$ is non-simple.

Some of the results proven above for the semigroup $S$ remain true for $S^\ast$.

\begin{lemma}
\label{l:exists_dist_term_nonsimplenew}
Suppose the matrix $\P$ of a finite simple semigroup $S=(G,\P,\Lambda,I)$ is nonsingular. Hence, for any pair of distinct elements $s_1,s_2\in S^\ast$ there exists a term $t(x)\in \T(S^\ast,\Gamma)$ separating $s_1,s_2$.
\end{lemma}
\begin{proof}
If the both elements $s_1,s_2$ belong to $S$ one can use the proof of Lemma~\ref{l:exists_dist_term}. Let now $s_1=\one$, $s_2=(\lambda,g,i)\in S$.

Consider the following cases.
\begin{enumerate}
\item Suppose $g\neq 1$, hence the elements $\one,(\lambda,g,i)$ are separated by the term $t(x)=(1,1,1)x(1,1,1)$:
\[
(1,1,1)\one(1,1,1)=(1,1,1)(1,1,1)=(1,1,1),
\]
\[
(1,1,1)(\lambda,g,i)(1,1,1)=(1,g,i)(1,1,1)=(1,g,1).
\] 
\item Suppose $g=1$ and $i\neq 1$. The non-singularity of $\P$ implies the existence if an index $\mu\in\Lambda$ such that $p_{i\mu}\neq 1$. Finally, the elements $\one,(\lambda,1,i)$ are separated by $t(x)=(1,1,1)x(\mu,1,1)$:
\[
(1,1,1)\one(\mu,1,1)=(1,1,1)(\mu,1,1)=(1,1,1),\\
\]
\[
(1,1,1)(\lambda,1,i)(\mu,1,1)=(1,1,i)(\mu,1,1)=(1,p_{i\mu},1)
\]
\item Assume $g=1$, $i=1$, $\lambda\neq 1$. As $\P$ is nonsingular, there exists an index $j\in I$ with $p_{j\lambda}\neq 1$. Hence, one can separate the elements $\one,(\lambda,1,1)$ by the term $t(x)=(1,1,j)x(1,1,1)$:
\[
(1,1,j)\one(1,1,1)=(1,1,j)(1,1,1)=(1,1,1),
\]
\[
(1,1,j)(\lambda,1,1)(1,1,1)=(1,p_{j\lambda},1)(1,1,1)=(1,p_{j\lambda},1).
\]
\item Finally, $s_2=(1,1,1)$. The non-singularity of $\P$ gives us the indexes $j\in I$, $\mu\in\Lambda$ such that $p_{j\mu}\neq 1$.  Thus, the elements $\one,(1,1,1)$ can be distinguished by $t(x)=(1,1,j)x(\mu,1,1)$:
\[
(1,1,j)\one(\mu,1,1)=(1,1,j)(\mu,1,1)=(1,p_{j\mu},1),
\]
\[
(1,1,j)(1,1,1)(\mu,1,1)=(1,1,1)(\mu,1,1)=(1,1,1).
\]
\end{enumerate}
\end{proof}

One can easily prove the analogs of Lemmas~\ref{l:singular->not_ED},~\ref{l:S-eq_dom->G-eq_dom},~\ref{l:sufficient_conditions} for the semigroup $S^\ast$, since their proofs are close to the corresponding results for $S$. Thus, we obtain the next result for the semigroup $S^\ast$.

\begin{theorem}
\label{th:about_nonsimple_domains}
Suppose a finite simple semigroup $S=(G,\P,\Lambda,I)$ is an e.d. in the language $\LL_S$, then so is $S^\ast$ in the language $\LL_{S^\ast}$.
\end{theorem}

\section{Semigroups with nonempty center}
\label{sec:center_semigroups}

\begin{theorem}
\label{th:right_left_ideal_center}
Let $I$ be a left (right) ideal of a semigroup $S$. Element $e\in S$ commutes with all elements of $I$, and there exists $a\in I$ such that $ea\neq a$.  Then the semigroup $S$ is not an e.d. in the language $\LL_S$.  
\end{theorem}
\begin{proof}
Let $I$ be a left ideal (similarly, one can consider the case, where $I$ is a right ideal).

Let us prove that the set $\M=\{(x_1,x_2)|x_1=a\mbox{ or }x_2=a\}$ is not algebraic over $S$. Assume the converse: there exists a system $\Ss(x_1,x_2)$ with the solution set $\M$, and $t(x_1,x_2)=s(x_1,x_2)$ is an equation of $\Ss$ which does not satisfy the point $(ea,ea)$.

Suppose that the equation $t(x_1,x_2)=s(x_1,x_2)$ has the form
\begin{multline*}
c_1w_1(x_1,x_2)c_2w_2(x_1,x_2)\ldots c_nw_n(x_1,x_2)c_{n+1}=\\
d_1u_1(x_1,x_2)d_2u_2(x_1,x_2)\ldots d_mu_m(x_1,x_2)d_{m+1},
\end{multline*}

where $w_i(x_1,x_2)$, $u_i(x_1,x_2)$ are coefficient-free. Denote by $t^\pr(x_1,x_2)$, $s^\pr(x_1,x_2)$ the expressions
\[t^\pr(x_1,x_2)=c_1w_1(x_1,x_2)c_2w_2(x_1,x_2)\ldots c_{n-1}w_{n-1}(x_1,x_2)c_nw_n(x_1,x_2),\]
\[s^\pr(x_1,x_2)=d_1u_1(x_1,x_2)d_2u_2(x_1,x_2)\ldots d_{m-1}u_{m-1}(x_1,x_2)d_mu_m(x_1,x_2).\]

As $(a,a)\in\M$, we have the equality 
\begin{equation}
\label{eq:q}
A=t^\pr(a,a)c_{n+1}=s^\pr(a,a)d_{m+1}.
\end{equation}

Denote by $n_i$ ($m_i$) the number of occurrences of the variable $x_i$ in the term $t^\pr(x_1,x_2)$ ($s^\pr(x_1,x_2)$).

Consider the calculation of the value $t^\pr(ea,a)$.
As the expressions $w_i(ea,a)$ contain the element $a\in I$ and the commuting element $e$, all occurrences of $e$ may be collected in the right part of any expression $w_i(ea,a)=w_i(a,a)e^{k_i}$, where $k_i$ is the number of occurrences of the variable $x_1$ in $w_i(x_1,x_2)$. Thus, the expression $t^\pr(ea,a)$ is written as
\[t^\pr(ea,a)=c_1w_1(a,a)e^{k_1}c_2w_2(a,a)e^{k_2}\ldots c_{n-1}w_{n-1}(a,a)e^{k_{n-1}}c_nw_n(a,a)e^{k_n}.\]
Since $c_nw_n(a,a)\in I$, the element $e$ commutes with $c_nw_n(a,a)$. Hence,
\[t^\pr(ea,a)=c_1w_1(a,a)e^{k_1}c_2w_2(a,a)e^{k_2}\ldots c_{n-1}w_{n-1}(a,a)c_nw_n(a,a)e^{k_{n-1}}e^{k_n}.\]
Similarly, all occurrences of $e$ maybe collected in the end of the expression
\begin{multline*}
t^\pr(ea,a)=c_1w_1(a,a)c_2w_2(a,a)\ldots c_{n-1}w_{n-1}(a,a)c_nw_n(a,a)e^{k_1}e^{k_2}\ldots e^{k_{n-1}}e^{k_n}=\\
c_1w_1(a,a)c_2w_2(a,a)\ldots c_{n-1}w_{n-1}(a,a)c_nw_n(a,a)e^{n_1}=t^\pr(a,a)e^{n_1}.
\end{multline*}

As the element $t^\pr(a,a)$ belongs to the ideal $I$, the occurrences of $e$ can be rewritten in the origin of the expression
\[t(ea,a)=e^{n_1}t^\pr(a,a)c_{n+1}.\]

Similarly, one can obtain the equality 
\[s(ea,a)=e^{m_1}s^\pr(a,a)c_{m+1}.\]

Since $(ea,a)\in\M$, we have the equality
\begin{equation}
\label{eq:qq}
e^{n_1}t^\pr(ea,a)c_{n+1}=e^{m_1}s^\pr(a,a)d_{m+1}.
\end{equation}

Analogically, one can prove that the points $(a,ea)\in\M$, $(ea,ea)\notin\M$ imply respectively the equality 
\begin{equation}
\label{eq:qqq}
e^{n_2}t^\pr(a,a)c_{n+1}=e^{m_2}s^\pr(a,a)d_{m+1}
\end{equation}
and the inequality
\begin{equation}
\label{eq:qqqq}
e^{n_1+n_2}t^\pr(a,a)c_{n+1}\neq e^{m_1+m_2}s^\pr(a,a)d_{m+1}.
\end{equation}

Using~(\ref{eq:q},\ref{eq:qq},\ref{eq:qqq}), we have
\begin{multline*}
e^{n_1+n_2}t^\pr(a,a)c_{n+1}=e^{n_2}(e^{n_1}A)=e^{n_2}(e^{m_1}A)=e^{m_1}(e^{n_2}A)=
e^{m_1}e^{m_2}A=\\
e^{m_1+m_2}s^\pr(a,a)d_{m+1},
\end{multline*}
that contradicts with~(\ref{eq:qqqq}).
\end{proof}

The following result follows easily from Theorem~\ref{th:right_left_ideal_center} if put $I=S$.

\begin{theorem}
\label{th:center_semigroups}
Any semigroup $S$ with a central element $e$ is not an e.d. in the language $\LL_S$ if there exists $a\in S$ such that $ae\neq a$.
\end{theorem}

\begin{corollary}
\label{cor:center_semigroups}
The following holds:
\begin{enumerate}
\item any nontrivial semigroup $S$ with zero is not an e.d. in the language $\LL_S$;
\item any nontrivial commutative semigroup $S$ is not an e.d. in the language $\LL_S$;
\item if a homogroup $S$ is an e.d. then $S$ is a group (i.d. $S=Ker(S)$).
\end{enumerate}
\end{corollary}
\begin{proof}
The proofs of the first two statements trivially follow from Theorem~\ref{th:center_semigroups}. Let us prove the last one.

Let $e$ be the identity element of the kernel $Ker(S)$ of a homogroup $S$. By Theorem~\ref{th:homogroups_properties}, the element $e$ belongs to  the center of $S$. If we assume the existence of an element $a\in S\setminus Ker(S)$, then $ae\in Ker(S)$, and hence $ae\neq a$. By Theorem~\ref{th:center_semigroups}, $S$ is not an e.d. in the language $\LL_S$.
\end{proof}

The information of the author:

Artem N. Shevlyakov

Omsk Branch of Institute of Mathematics, Siberian Branch of the Russian Academy of Sciences

644099 Russia, Omsk, Pevtsova st. 13

Phone: +7-3812-23-25-51.

e-mail: \texttt{a\_shevl@mail.ru}
\end{document}